\documentclass[10pt,openany]{article}
\usepackage[utf8]{inputenc}
\usepackage[a4paper,hmargin={4cm,4cm},vmargin={3.5cm,3.5cm}]{geometry}
\usepackage[colorlinks=true,bookmarksopen=false,linkcolor=blue,citecolor=red]{hyperref}
\usepackage{calc} 
\usepackage{datetime}
\usepackage{comment}
\usepackage{amssymb}
\usepackage{amsthm}
\usepackage{amsmath}
\usepackage{amsrefs}
\usepackage{titlesec}
\usepackage{titletoc}
\usepackage{dsfont}
\usepackage{euscript}
\usepackage{fourier-orns}
\usepackage{palatino}
\usepackage[sc]{mathpazo} 
\usepackage[T1]{fontenc}
\usepackage{graphicx}
\usepackage{tikz}

\definecolor{blue}{rgb}{0, 0.1, 0.6}

\DeclareFontFamily{OT1}{pzc}{}
\DeclareFontShape{OT1}{pzc}{m}{it}{<-> s * [1.10] pzcmi7t}{}
\DeclareMathAlphabet{\mathpzc}{OT1}{pzc}{m}{it}

\renewenvironment{itemize}
  {\begin{list}{$\triangleright$}{%
   \setlength{\parskip}{0mm}
   \setlength{\topsep}{.4\baselineskip}
   \setlength{\rightmargin}{0mm}
   \setlength{\listparindent}{0mm}
   \setlength{\itemindent}{0mm}
   \setlength{\labelwidth}{3ex}
   \setlength{\itemsep}{.4\baselineskip}
   \setlength{\parsep}{0mm}
   \setlength{\partopsep}{0mm}
   \setlength{\labelsep}{1ex}
   \setlength{\leftmargin}{\labelwidth+\labelsep}
   }}{%
   \end{list}\vspace*{-\parskip}}

\def\E{\exists}

\def\models{\vDash}
\def\pmodels{\mathrel{\models\kern-1.5ex\raisebox{.5ex}{*}}}

\def\RR{\mathds R}

\def\supp{\mathop{\rm supp}}

\def\fp{{\rm fp}}

\def\Th{\textrm{Th}}

\def\Aut{\textrm{Aut\kern.15ex}}
\def\Autf{\mathord{\rm Aut\kern.15ex{f}\kern.15ex}}
\def\Cb{\textrm{Cb\kern.15ex}}

\def\id{{\rm id}}
\def\tp{{\rm tp}}

\newcommand{\cev}[1]{\reflectbox{\ensuremath{\vec{\reflectbox{\ensuremath{#1}}}}}}

\def\nonfork{\mathop{\raise0.2ex\hbox{
   \ooalign{\hidewidth$\vert$\hidewidth\cr\raise-0.9ex\hbox{$\smile$}}}}}

\def\cnonfork{\mathbin{\raise1.8ex\rlap{\kern0.6ex\rule{0.6ex}{0.1ex}}
\rlap{\kern1.1ex\rule{0.1ex}{1.9ex}}\raise-0.3ex\hbox{$\smile$} } }

\def\cpaw{\mathbin{\ooalign{\kern-0.4ex$-$\hidewidth\cr$<$}}}
\def\cpawdot{\ooalign{$\kern1.2ex\cdot$\cr$\cpaw$\cr}}

\def\sm{\smallsetminus}

\def\IMP{\Rightarrow}

\def\IFF{\Leftrightarrow}

\def\isomap{\rlap{\kern0.8ex\raisebox{1ex}{\scriptsize$\sim$}}\rightarrow}

\def\Aa{\EuScript A}

\def\C{\EuScript C}
\def\U{\EuScript U}

\def\G{\EuScript G}
\def\B{\EuScript B}

\def\O{\EuScript O}

\def\<{\langle}
\def\>{\rangle}
\def\0{\varnothing}
\def\theta{\vartheta}
\def\phi{\varphi}
\def\epsilon{\varepsilon}
\def\ssf#1{\textsf{\small #1}}

\titlecontents{section}
[3.8em] 
{\vskip-1ex}
{\contentslabel{1.5em}}
{\hspace*{-2.3em}}
{\titlerule*[1pc]{}\contentspage}

\titleformat{\section}[block]{\Large\bfseries}{\makebox[5ex][r]{\textbf{\thesection}}}{1.5ex}{}
\titlespacing*{\chapter}{0em}{.5ex plus .2ex minus .2ex}{2.3ex plus .2ex}
\titlespacing*{\section}{-9.7ex}{3ex plus .5ex minus .5ex}{1ex plus .2ex minus .2ex}

\newtheoremstyle{mio}
     {2\parskip}
     {\parskip}
     {\sl}
     {}
     {\bfseries}
     {}
     {1ex}
     {\llap{\thmnumber{#2}\hskip2mm}
      \thmname{#1}\thmnote{\bfseries{} #3}}

\newtheoremstyle{liscio}
     {2\parskip}
     {0mm}
     {}
     {}
     {\bfseries}
     {}
     {1.5ex}
     {\llap{\thmnumber{#2}\hskip2mm}
      \thmname{#1}\thmnote{\bfseries{} #3}}

\newcounter{thm}[section]

\theoremstyle{mio}

\newtheorem{corollary}[thm]{Corollary}
\newtheorem{proposition}[thm]{Proposition}
\newtheorem{lemma}[thm]{Lemma}

\newtheorem{definition}[thm]{Definition}

\theoremstyle{liscio}

\newtheorem{remark}[thm]{Remark}
\newtheorem{notation}[thm]{Notation}

\newtheorem{example}[thm]{Example}
\def\QED{\noindent\nolinebreak[4]\hspace{\stretch{1}}\rlap{\ \ $\Box$}\medskip}
\renewenvironment{proof}[1][Proof]%
{\begin{trivlist}\item[\hskip\labelsep {\bf #1}]}
{\QED\end{trivlist}}

{\begin{trivlist}\item[\hskip\labelsep {\bf #1}]}
{\QED\end{trivlist}}

\definecolor{violet}{RGB}{115, 0, 205}
\definecolor{brown}{RGB}{150, 50, 10}
\definecolor{green}{RGB}{5,110, 35}
\definecolor{emphcolor}{rgb}{.98,.98,.70}

\def\bl{\color{black}}
\def\mr{\color{brown}}

\def\mrA{{\mr\Aa}}
\def\mrB{{\mr\B}}
\def\mrC{{\mr\C}}

\def\mrG{{\mr\G}}
\def\mrU{{\mr\U}}

\renewcommand*{\emph}[1]{%
   \kern-0.6ex 
   \smash{\tikz[baseline]
   \node[ rectangle,       fill=emphcolor,  rounded corners, 
          inner xsep=.9ex, inner ysep=.2ex, anchor=base,
          minimum height = 3ex
         ]{#1};
   }
   \kern-1.7ex 
}

\begin{document}
\raggedbottom
\begin{center}
   {\huge\bfseries Ramsey's coheirs\\[3ex] \normalfont\normalsize 
   Eugenio Colla and Domenico Zambella\vskip-1ex 
   Università di Torino\vskip-1ex \monthname\ \the\year}
\end{center}

\def\medrel#1{\parbox[t]{6ex}{$\displaystyle\hfil #1$}}

\bigskip\hfil
\parbox{0.9\textwidth}{
   \textbf{Abstract} \ 
   We use the model theoretic notion of coheir to give
   short proofs of old and new theorems in Ramsey Theory.
   As an illustration we start from Ramsey's theorem itself.
   Then we prove Hindman's theorem and the Hales-Jewett theorem.
   Finally, we prove two Ramsey theoretic principles that have among 
   their consequences partition theorems due to Carlson and to Gowers.
   
   Msc: primary 05D10, secondary 03C98, 03H99.
}
\section{Introduction}\label{intro}

Ramsey theory has substantial and diverse applications 
to many parts of mathematics.
In particular, Ramsey's theorem has foundational applications 
to model theory through the Ehrenfeucth-Mostowski construction 
of indiscernibles and generalizations thereof.
In this paper we explore the converse direction, that is, we use model theory to obtain new proofs of classical results in Ramsey Theory.

The Stone-\v{C}ech compactification, obtained via ultrafilters, is a widely employed method for proving Ramsey theoretic results.
One of its first major applications is the celebrated Galvin-Glazer proof of Hindman's theorem, see e.g.\@ \cite{Blass}.
Our methods are related, but alternative, to the ultrafilter
approach.
We replace $\beta G$ (the Stone-\v{C}ech compactification 
of a semigroup $G$) with a large saturated elementary extension of $G$,
i.e.\@ a monster model of $\Th(G/G)$.
One immediate advantage is that we work with elements of a natural semigroup
with a natural operation.
In contrast, elements of $\beta G$ are ultrafilters, that is, sets of sets, 
and the semigroup operation among ultrafilters is far from straightforward.

This idea is not completely new:
in his seminal work on the applications of topological 
dynamics to model theory\@ \cite{Newelski1, Newelski2}, Newelski 
replaces the semigroup $\beta G$ with the space of types over $G$ with 
a suitably defined operation.
Our approach is similar, except that, unlike Newelski, we do not pursue 
connections with topological dynamics, but rather offer an alternative 
realm of application.
The investigation of alternative methods in the study of regularity phenomena 
has been called for by Di Nasso \cite[Open problem~\#1]{Mauro2}.
This article contains a possible answer.

The model theoretic tools employed in this paper are relatively basic.
Section~\ref{coheirs} is meant to give  an accessible overview of the necessary 
notions for readers whose expertise is not primarily in model theory. 
Our results do not require assumptions of model theoretic tameness such as 
stability, NIP, etc., much like those that use nonstandard analysis, 
for example in \cite{Mauro}.
Investigating the effect of such assumptions remains as future work.

The second author is grateful to Pierre Simon for suggesting the comparison 
with nonstandard analysis.
Both authors would like to thank Vassilis Kanellopoulos for helpful conversation.
When this paper was essentially complete, we became aware of~\cite{ACG}, which 
is worth mentioning since it employs similar methods in a related context.

\hfil ***
\\
The paper is divided into two parts.
In the first part we prove that the notion of coheir 
leads to short and elegant proofs of well-known results.
Most proofs in this part may be considered folklore, though they have not appeared in the literature so far.
They are included here to provide a self-contained, gentle 
introduction to the techniques that are used in the second 
part. 

As a preliminary illustrative step, we present a proof of 
Ramsey's theorem (Theorem~\ref{thm_Ramsey}).
Then we prove a generalization of Hindman's theorem 
(Theorem~\ref{thm_Hindman}), which is
 required in the second part of the paper.
We also show how to combine Ramsey's and Hindman's theorems in a single
proposition -- the Milliken–Taylor theorem (Theorem~\ref{thm_MillikenTaylor}).
Finally, we prove an abstract algebraic version of the Hales-Jewett theorem
(Theorem~\ref{thm_abstract_HJ}) due to Sabine Koppelberg~\cite{Koppelberg}.

In the second part of the paper we prove two Ramsey-theoretic properties
of semigroups (Lemmas~\ref{lem_Carlson} and~\ref{lem_Gowers}).
As an application, we derive a generalization of Carlson's theorem on colourings of variable
words which we present in the style of Koppelberg (Theorem~\ref{thm_Carlson2})
and in its classical form (Corollary~\ref{corol_Carlson}).
Lemma~\ref{lem_Gowers} is a partition theorem that generalizes Gowers's  FIN$_k$ Theorem~\cite{Gowers} in a different direction than~\cite{Lupini}.

The extent of the generalizations mentioned above is limited, and they could be obtained in other ways, but our motivation here is to show the use and relevance of model theoretic methods.
Numerous papers in the literature strengthen or generalize the partition theorems considered here.
The comparison of the results that appear in these papers is not always straightforward -- a few are compared in \cite{AC}. 

\hfil ***
\\
The proofs in this paper require a modicum of familiarity with model theory.
However, the results can be stated in an elementary language, and in the rest of this introduction we introduce the necessary terminology.

Throughout the paper $G$ is a semigroup and 
$\Sigma$ a non-empty set of endomorphisms of $G$.
For ${\bar a}\in G^{\le\omega}$ we write

\def\ceq#1#2#3{\parbox[t]{8ex}{$\displaystyle #1$}\medrel{#2}{$\displaystyle #3$}}

\ceq{\hfill\emph{$\fp^{\Sigma}\,{\bar a}$} }
{=}
{\Big\{\sigma_{0}\,{a_{i_0}}\kern-0.5ex\cdot\dots\cdot\sigma_{k}\,{a_{i_k}}
\ :\  
i_0<\dots<i_k<|{\bar a}|, \ \  
\bar\sigma\in\big(\Sigma\cup\{\id_G\}\big)^{k+1}, \ 
k<|{\bar a}|\Big\}}

Overlined symbols, such as $\bar a$ or $\bar\sigma$, always denote a 
tuple, and  $a_i$, $\sigma_i$ denotes the $i$-th entry of that tuple.

When $\Sigma$ is empty, we write \emph{$\fp\,{\bar a}$.}

\medskip
\begin{example}\label{ex_words}
   For future reference, we instantiate the definition above in the 
   context of free semigroups.
   Let $G$ be the set of words on a finite alphabet $A\cup\{x\}$, 
   where $x$ is a symbol not in $A$ which we call \emph{variable.}
   Let $C$ be the set of words on the alphabet $A$.
   Words in $C$ are called \emph{constant words}, while 
   those in $G\sm C$ are called \emph{variable words.}
   When $G$ is endowed with the operation of concatenation of words, 
   $C$ and $G\sm C$ are subsemigroups of $G$.
   For $t\in G$ and $a\in A$, let $t(a)$ be the word obtained by replacing 
    all the occurrences of $x$ in $t$ by $a$.
   Note that the map $\sigma_a:t\mapsto t(a)$ is an endomorphism of $G$.
   In the literature, when $G$ is as above and $\Sigma=\{\sigma_a:a\in A\}$,
   the elements of $\fp^{\Sigma}\,{\bar s}$ are called \emph{extracted words.}
   We say that a tuple $\bar a\in\big(\fp^{\Sigma}\bar s\big)^{\omega}$ is an \emph{extracted sequence\/} if $a_i\in\fp^{\Sigma}s_{\restriction\, [n_i,n_{i+1})}$ for some increasing sequence of positive integers $\<n_i:i<\omega\>$.
   If, moreover, $a_i\notin C$ for all $i$, we say that $\bar a$ is an \emph{extracted variable sequence\/} of $\bar s$.\QED
\end{example}

The following definition will be used to express our results in the
general context of semigroups.

\begin{definition}
   Let \emph{$\cpaw$\/} be a binary relation on $G$.
   We say that $G$ is \emph{$\cpaw$-covered\/} if for every finite
   $A\subseteq G$ there is a $c$ such that $A\cpaw{c}$.
   If $c$ can be found in some fixed $B\subseteq G$, we say $\cpaw$-covered \emph{by $B$\/}.
   We say that $G$ is \emph{\cpawdot-closed\/} if ${a}\cpaw{b}\cpaw{c}$ 
   implies ${a}\cpaw{b}{\cdot}{c}$ for all ${a},{b},{c}\in G$.
   A \emph{$\cpaw$-chain\/} in $G$ is a tuple ${\bar a}\in G^{\le\omega}$ 
   such that ${a_i}\cpaw{a_{i+1}}$.
\end{definition}

The preorder relation given by the length of the words on a free semigroup $G$
is a natural example that is both \cpawdot-closed and $\cpaw$-covered.
A less straightforward relation is used in the proof of 
Theorem~\ref{thm_Carlson2}.

Finally, we recall two standard notions.
Let $C\subseteq G$ be a subsemigroup.
We say that $C$ is \emph{nice\/} if $a\cdot b\in C$ implies $a, b\in C$.
A homomorphism $\sigma:G\to C$ such that $\sigma_{\restriction C}=\id_C$ 
is called \emph{retraction\/} of $G$ onto $C$.
Note that the set of constant words in Example~\ref{ex_words} is a nice 
subsemigroup and that the maps $\sigma_a$ are retractions.

We are now ready to state Lemma~\ref{lem_Carlson}.

\textbf{Lemma}\ 
   Let $\Sigma$ be a finite set of retractions of $G$ onto a nice 
   subsemigroup $C$.
   Let $\cpaw$ be a relation on $G$ that makes it \cpawdot-closed 
   and $\cpaw$-covered by $G\sm C$.
   Then, for every finite coloring of $G$, there is a $\cpaw$-chain 
   $\bar a\in (G\sm C)^\omega$ such that $\fp^{\Sigma}\bar a\sm C$
   is monochromatic.
\QED

When $C$ and $\Sigma$ are empty and $\cpaw$ holds for all pairs, 
the lemma reduces to Hindman's theorem (Theorem~\ref{thm_Hindman}).

The appropriate choice of $G$, $C$, $\Sigma$ and $\cpaw$ yields Carlson's 
partition theorem (in particular no model theoretic argument is 
necessary, see Theorem~\ref{thm_Carlson2} and its 
Corollary~\ref{corol_Carlson}).

In the last section we prove Lemma~\ref{lem_Gowers} which is similar to the lemma above but deals with composition of homomorphisms.
This is also stated in an elementary language and a general version of a
partition theorem by Gowers is derived from it.

\section{Coheirs, and coheir sequences}\label{coheirs}

We assume that the reader is familiar with undergraduate model theory
and in this section we only review the few prerequisites that go beyond that.
Proofs are omitted.
The reader may consult any standard model theory textbook e.g.\@ \cite{TZ} 
(the intrepid reader may consult \cite{DZ}, some lecture notes which 
use the same notation and quirks as this paper).
The notation and terminology are standard with the possible exception of 
Definitions~\ref{def_coheir_idepencence} and~\ref{def_coheir_stationary}.

A \emph{sequence\/} is a function whose domain is a linear order.
A \emph{tuple\/} is a sequence whose domain is an ordinal.
The domain of the tuple ${\mr c}$ is denoted by \emph{$|{\mr c}|$\/} and 
is called the \emph{length\/} of ${\mr c}$.

\begin{notation}\label{notation1}
   Sometimes (i.e.\@ not always) we may overline tuples as mnemonic.
   When a tuple ${\mr\bar c}$ is introduced, ${\mr c_i}$ denotes 
   the $i$-th element of ${\mr\bar c}$.
   We write ${\mr c_{\restriction I}}$, where $I\subseteq |{\mr\bar c}|$, 
   for the tuple which is naturally associated to the restriction of 
   ${\mr\bar c}$ to $I$.
   The bar is dropped for ease of notation.\QED
\end{notation}

We denote the monster model by $\mrU$ or, when dealing with semigroups, 
by $\mrG$.
We always work over a fixed set of parameters $A\subseteq\mrU$.
When this set is a model, as it will often be, we denote it by $M$,
or $G$ in the case of semigroups.

We say that a type $p({\bl x})$ is  \emph{finitely satisfied\/} in $A$
if every conjunction of formulas in $p({\bl x})$ has a solution 
in $A^{|{\bl x}|}$.
A global type that is finitely satisfiable in $A$ is invariant over $A$.

If $M$ is a model every consistent type 
$p({\bl x})\subseteq L(M)$ is finitely satisfied in $M$.
For this reason in a few points in this paper it is 
necessary to work over a model.
For simplicity, we always assume this.

The following is an easy, well-known fact.

\begin{proposition}\label{prop_exisntence_coheirs}
   Every type $q({\bl x})\subseteq L(\mrU)$ that is finitely 
   satisfiable in $M$ has an extension to a global type finitely 
   satisfiable in $M$.\QED
\end{proposition}

If $p({\bl x})$ is finitely satisfied in $M$, the extensions of 
$p({\bl x})$ that are also finitely satisfied in $M$ are called 
\emph{coheirs\/} of $p({\bl x})$.

In many cases it is useful to focus on elements instead of their types.
We introduce the following notation to express that $\tp({\mr a}/M,{\mr b})$ is finitely satisfied in $M$.
(The notion is standard in model-theory, it has no standard notation though.)

\def\ceq#1#2#3{\parbox[t]{15ex}{$\displaystyle #1$}
               \medrel{#2}
               {$\displaystyle #3$}}

\begin{definition}\label{def_coheir_idepencence}
   For every ${\mr a}\in\mrU^{|{\bl x}|}$ and ${\mr b}\in\mrU^{|{\bl z}|}$ 
   we define

   \ceq{\hfill\emph{${\mr a}\cnonfork_M{\mr b}$}}
   {\IFF}
   {\phi({\mr a}\,;{\mr b})
   \textrm{ for all }\phi({\bl x}\,;{\bl z})\in L(M)  
   \textrm{ such that }M^{|{\bl x}|}\subseteq\phi(\mrU^{|{\bl x}|}\,;{\mr b})}

   We call this the \emph{coheir-heir} relation. We define the type

   \ceq{\hfill\emph{${\bl x}\cnonfork_M{\mr b}$}}
   {=}
   {\Big\{\phi({\bl x}\,;{\mr b})
   \ :\ 
   \phi({\bl x}\,;{\mr b})\in L(M)
   \textrm{ and } M^{|{\bl x}|}\subseteq\phi(\mrU^{|{\bl x}|}\,;{\mr b})\Big\}.}

   The tuples ${\mr a}$ realizing this type are those such that ${\mr a}\cnonfork_M{\mr b}$. We will use the symbol \emph{${\mr a}\equiv_A{\bl x}\cnonfork_M{\mr b}$} 
   for the union of the types ${\bl x}\cnonfork_M{\mr b}$ and 
   $\tp({\mr a}/M)$.\QED
\end{definition}

We imagine ${\mr a}\cnonfork_M{\mr b}$ as saying that ${\mr a}$ is 
\emph{independent\/} from ${\mr b}$ over $M$.
This is a very strong form of independence.
In general it is not symmetric, that is, ${\mr a}\cnonfork_M{\mr b}$ is not
the same as ${\mr b}\cnonfork_M{\mr a}$ (symmetry is equivalent to stability).

We shall use, sometimes without reference, the following easy lemma.

\begin{lemma}\label{lem_coheir_independence}
The following properties hold for all small $M,{\mr a},{\mr b}$, and ${\mr c}$
\begin{itemize}
\item[1.] ${\mr a}\cnonfork_M{\mr b}\ \ \IMP\ \ f{\mr a}\cnonfork_Mf{\mr b}$ \ \ 
          for every $f\in\Aut(\mrU/M)$\hfill \textit{invariance}
\item[2.] ${\mr a}\cnonfork_M{\mr b}\ \ \IFF\ \ {\mr a_0}\cnonfork_M{\mr b_0}$
          \ for all finite ${\mr a_0}\subseteq{\mr a}$ and 
          ${\mr b_0}\subseteq{\mr b}$ \hfill\textit{finite character}
\item[3.] ${\mr a}\cnonfork_M{\mr b},{\mr c}$ \ and \ 
          ${\mr b}\cnonfork_M{\mr c}\ \ \IMP\ \ {\mr a},{\mr b}\cnonfork_M{\mr c}$
          \hfill\hfill\hfill\textit{transitivity}
\item[4.] ${\mr a}\cnonfork_M{\mr b}\ \ \IMP\ \ $ 
          there exists ${\mr a'}\equiv_{M,\,{\mr b}}{\mr a}$ such that 
          ${\mr a'}\cnonfork_M{\mr b},{\mr c}$
          \hspace{\stretch{20}}\textit{coheir extension}\QED
\end{itemize}
\end{lemma}

Note that ${\mr a}\equiv_M{\bl x}\cnonfork_M{\mr b}$ is the intersection of all types in $S(M, {\mr b})$ that are coheirs 
of $\tp({\mr a}/M)$.
As there may be more than one of such coheirs, ${\mr a}\equiv_M{\bl x}\cnonfork_M{\mr b}$ need not be a complete over $M,{\mr b}$.
In fact, completeness is a rather strong property.

\begin{definition}\label{def_coheir_stationary}
   If ${\mr a}\equiv_M{\bl x}\cnonfork_M{\mr b}$ is a complete type 
   (over $M,{\mr b}$) for every ${\mr a}\in\mrU^{<\omega}$, every 
   ${\mr b}\in\mrU^{|{\bl x}|}$, and every tuple of variables ${\bl x}$,
   then we say that $\cnonfork_M$ is \emph{stationary.}
   We say \emph{$n$-stationary\/} if the requirement above is 
   restricted to $|{\bl x}|=n$.\QED
\end{definition}

Stationarity is often ensured by the following property.

\begin{proposition}
   Fix a tuple of variable ${\bl x}$ of length $n$.
   If for every $\phi({\bl x})\in L(\mrU)$ there is a formula 
   $\psi({\bl x})\in L(M)$ such that $\phi(M^{|{\bl x}|})=\psi(M^{|{\bl x}|})$
   then $\cnonfork_M$ is $n$-stationary.\QED
\end{proposition}

\begin{remark}\label{rk_coheir_stationary}
   Stationarity of $\cnonfork_M$ over every model $M$ 
   is equivalent to the stability of $T$.
   However, in unstable theories the assumption may hold 
   for some particular model.
   For example, if every subset of $M^n$ is the trace of a definable set,
   then $\cnonfork_M$ is $n$-stationary by the proposition above.
   This simple observation will be of help in the proof of 
   Theorem~\ref{thm_Hindman}.
   For natural example let $T=T_{\rm dlo}$ and let $M\subseteq\mrU$
   have the order-type of $\RR$.
   By quantifier elimination every definable of $\mrU$ is union of finitely
   many intervals.
   By Dedekind completeness, the trace on $A$ of any interval of $\mrU$ coincides 
   with that of an $M$-definable interval.\QED
\end{remark}

Let \mbox{$p({\bl x})\in S(\mrU)$} be a global type that is finitely 
satisfiable in $M$.
We say that the tuple ${\mr\bar c}$ is a \emph{coheir sequence\/} of
$p({\bl x})$ over $M$ if for every $i<|{\mr\bar c}|$

\ceq{\hfill {\mr c_i}}
   {\models}
   {p_{\restriction M,\,{\mr c_{\restriction i}}}({\bl x})}.

The following is a convenient characterization of coheir sequences.

\begin{lemma}\label{lem_coheir_property}
   For ${\mr\bar c}$ a tuple of length $\omega$, the following are equivalent 
   \begin{itemize}
   \item[1.] ${\mr\bar c}$ is a coheir sequence over $M$;
   \item[2.] ${\mr c_n}\cnonfork_M{\mr c_{\restriction n}}$ and 
             ${\mr c_{n+1}}\equiv_{M,\,{\mr c_{\restriction n}}}{\mr c_n}$ 
             for every $n<\omega$.\QED
   \end{itemize}
\end{lemma}

Let $I,<_I$ be a linear order.
We call a function ${\mr\bar a}:I\to\mrU^{|{\bl x}|}$ an \emph{$I$-sequence}, 
or simply a \emph{sequence\/} when $I$ is clear.

If $I_0\subseteq I$ we call ${\mr a_{\restriction I_0}}$, the restriction of 
${\mr\bar a}$ to $I_0$, a \emph{subsequence\/} of ${\mr\bar a}$.
When $I_0$ is finite we identify ${\mr a_{\restriction I_0}}$ 
with a tuple of length $|I_0|$.
 
\begin{definition}
   Let $I,<_I$ be an infinite linear order and let ${\mr\bar a}$ be an 
   $I$-sequence.
   We say that $a$ is a \emph{sequence of indiscernibles\/} over $A$ or, 
   a sequence of \emph{$A$-indiscernibles}, if 
   ${\mr a_{\restriction I_0}}\equiv_A {\mr a_{\restriction I_1}}$ 
   for every $I_0,I_1\subseteq I$ of equal finite cardinality.\QED
\end{definition}

The following can be easily derived from the lemma above by induction.

\begin{proposition}
   Every sequence of coheirs over $M$ is $M$-indiscernible.\QED
\end{proposition}
\section{Ramsey's theorem from coheir sequences}\label{Ramsey}

We illustrate the relation between coheirs and Ramsey phenomena
in the simplest possible case: Ramsey's theorem.
The subsequent sections build on this proof for more sophisticated results.

In this chapter we deal with finite partitions.
The partition of a set $X$ into $k$ subsets is often represented by 
a map $f:X\to [k]$.
The elements of $[k]=\{1,\dots,k\}$ are also 
called \emph{colors}, and the partition a \emph{coloring},
or \emph{$k$-coloring}, of $X$.
We say that $Y\subseteq X$ is \emph{monochromatic\/} if 
$f$ is constant on $Y$.

Let $M$ be an arbitrary infinite set.
Fix $n,k<\omega$ and fix a coloring $f$ of the set of all 
\emph{$n$-subsets\/} of $M$, aleas the 
\emph{complete $n$-uniform hypergraph\/} with vertex set $M$,

\ceq{\hfill f}
   {:}
   { {M\choose n}\ \ \to\ \ [k]}.

We say that $H\subseteq M$ is a \emph{monochromatic subgraph\/} if 
the subgraph induced by $H$ is monochromatic.
In the literature monochromatic subgraphs are also called 
\emph{homogeneous sets.}

The following is a very famous theorem which we prove here 
in an unusual way.
The proof will serve as a blueprint for other constructions 
in this paper.

\theoremstyle{mio}
\newtheorem{Ramsey}[thm]{Ramsey's Theorem}
\begin{Ramsey}\label{thm_Ramsey}
   Let $M$ be an infinite set.
   Then for every positive integer $n$ and every finite coloring of 
   the complete $n$-uniform hypergraph with vertex set $M$
   there is an infinite monochromatic subgraph.
\end{Ramsey}

\begin{proof}
   Let $L$ be a language that contains $k$ relation symbols $r_1,\dots,r_k$
   of arity $n$.
   Given a $k$-coloring $f$ we define a structure with domain $M$.
   The interpretation of the relation symbols is

   \ceq{\hfill r_i^M}
      {=}
      {\Big\{  a_1,\dots,a_n\in M 
      \ : \ 
      f\big(\{a_1,\dots,a_n\}\big)= i\Big\}.} 

   We may assume that $M$ is an elementary substructure of some 
   large saturated model $\mrU$.
   Pick any type $p({\bl x})\in S(\mrU)$ finitely satisfied in $M$ 
   but not realized in $M$ and let ${\mr\bar c}=\<{\mr c_i}:i<\omega\>$ 
   be a coheir sequence of $p({\bl x})$.

   There is a first-order sentence saying that the formulas 
   $r_i({\bl x_1},\dots,{\bl x_n})$ are a coloring of ${M\choose n}$.
   Then by elementarity the same holds in $\mrU$.
   By indiscernibility, all tuples of $n$ distinct elements of 
   ${\mr\bar c}$ have the same color, say $1$.
   We now prove that there is a sequence $\bar a=\<a_i:i<\omega\>$ 
   in $M$ with the same property.
   
   We construct $a_{\restriction i}$ by induction on $i$ as follows.

   Assume as induction hypothesis that the subsequences of length $n$
   of $a_{\restriction i},{\mr c_{\restriction n}}$ all have color $1$.
   Our goal is to find $a_i\in M$ such that the same property holds 
   for $a_{\restriction i},a_i,{\mr c_{\restriction n}}$.
   By the indiscernibility of ${\mr\bar c}$, the property holds for 
   $a_{\restriction i},{\mr c_{\restriction n}},{\mr{c_n}}$.
   And this can be written by a formula 
   $\phi(a_{\restriction i},{\mr c_{\restriction n}},{\mr{c_n}})$.
   As ${\mr\bar c}$ is a coheir sequence, by Lemma~\ref{lem_coheir_property}
   we can find  $a_i\in M$ such that 
   $\phi(a_{\restriction i},{\mr c_{\restriction n}},a_i)$.
   So, as the order is irrelevant, 
   $a_{\restriction i},a_i,{\mr c_{\restriction n}}$ 
   satisfies the induction hypothesis.
\end{proof}
\section{Idempotent orbits in semigroups}\label{semigroups}

\def\medrel#1{\parbox[t]{6ex}{$\displaystyle\hfil #1$}}
\def\ceq#1#2#3{\parbox[t]{22ex}{$\displaystyle #1$}
              \medrel{#2}
              {$\displaystyle #3$}}

In this and the following sections we fix a semigroup $G$
which we identify with a first-order structure.
The language contains, among others, the symbol \emph{$\ \cdot\ $} 
which is interpreted as a binary associative operation on $G$.
We write \emph{$\mrG$} for a large saturated elementary extension of $G$.

For any two sets $\mrA,\mrB\subseteq\mrG$ we define

\ceq{\hfill\emph{$\mrA\cdot_G\mrB$}}
   {=}
   {\Big\{ {\mr a}{\cdot}{\mr b}
      \ :\ 
      {\mr a}\in\mrA, \ 
      {\mr b}\in\mrB\textrm{ and }\ 
      {\mr a}\cnonfork_G{\mr b}\Big\}}

In this and the next section we abbreviate $\O({\mr a}/G)$, 
the orbit of ${\mr a}$ under $\Aut(\mrG/G)$, 
with \emph{${\mr a}_G$}.
We write \emph{${\mr a}\cdot_G\mrB$} for $\O({\mr a}/G)\cdot_G\mrB$.
Similarly for \emph{$\mrA\cdot_G{\mr b}$} 
and \emph{${\mr a}\cdot_G{\mr b}$.}

\begin{lemma}\label{lem_typedef_Ab}
   If $\mrA$ is type definable over $G$ then so is $\mrA\cdot_G{\mr b}$ 
   for any ${\mr b}$.
\end{lemma}

\begin{proof} 
   The set  $\mrA\cdot_G{\mr b}$ is the union of $\mrA\cdot_G\{{\mr c}\}$ 
   as ${\mr c}$ ranges in ${\mr b}_G$.
   The set $\mrA\cdot_G\{{\mr c}\}$ is type definable, say by the 
   the type $\E{\bl y}\, p({\bl x}, {\bl y},{\mr c})$ where
   
   \ceq{\hfill p({\bl x}, {\bl y},  {\mr c})}
      {=}
      {{\bl y}\cnonfork_G{\mr c}
      \ \ \wedge\ \ 
      {\bl y}{\cdot}{\mr c}={\bl x}
      \ \ \wedge\ \ 
      {\bl y}\in\mrA}
   
   Note that, by the invariance of $\cnonfork_G$, if $f\in\Aut(\mrG/G)$,
   then $\E{\bl y}\, p({\bl x}, {\bl y},f{\mr c})$ defines 
   $\mrA\cdot_G\{f{\mr c}\}$.
   Therefore if $q({\bl z})=\tp({\mr b}/G)$ then 
   $\E{\bl y},{\bl z}\, \big[q({\bl z})\cup p({\bl x}, {\bl y},{\bl z})\big]$ 
   defines $\mrA\cdot_G {\mr b}$.
\end{proof}

By the invariance of $\cnonfork_G$, for every $f\in\Aut(\mrG/G)$
we have $f[\mrA\cdot_G\mrB]=f[\mrA]\cdot_Gf[\mrB]$.
Therefore when $\mrA$ and $\mrB$ are invariant over $G$,
also $\mrA\cdot_G\mrB$ is invariant over $G$.
Below we mainly deal with invariant sets.

\begin{proposition}\label{prop_semi_associative}
   For all $G$-invariant sets $\mrA$, $\mrB$, and  $\mrC$

   \ceq{\hfill\mrA\cdot_G\big(\mrB\cdot_G\mrC\big)}
      {\subseteq}
      {\big(\mrA\cdot_G\mrB\big)\cdot_G\mrC}.
\end{proposition}

\begin{proof}
   Let ${\mr a}{\cdot}{\mr b}{\cdot}{\mr c}$ be an arbitrary element of the
   l.h.s.\@ where ${\mr a}\cnonfork_G{\mr b}{\cdot}{\mr c}$ 
   and ${\mr b}\cnonfork_G{\mr c}$.
   By extension (Lemma~\ref{lem_coheir_independence}),
   there exists ${\mr a'}$ such that 
   ${\mr a}
      \equiv_{G,\,{\mr b}{\cdot}{\mr c}}
      {\mr a'}
      \cnonfork_G
      {\mr b}{\cdot}{\mr c},\,  {\mr b},\, {\mr c}$.
   By transitivity (again Lemma~\ref{lem_coheir_independence}),
   ${\mr a'}{\cdot}{\mr b}\cnonfork_G{\mr c}$.
   Therefore ${\mr a'}{\cdot}{\mr b}{\cdot}{\mr c}$ belongs to the r.h.s.
   Finally, as ${\mr a'}\equiv_{G,\,{\mr b}{\cdot}{\mr c}}{\mr a}$, also 
   ${\mr a}{\cdot}{\mr b}{\cdot}{\mr c}$ belongs to the r.h.s.\@ by invariance.
\end{proof}

Let $\mrA$ be a non-empty set.
When $\mrA\cdot_G\mrA\subseteq\mrA$, we say that it is 
\emph{idempotent\/} (over $G$).

\begin{corollary}\label{corol_min_idempotent}
   Assume $\mrB\subseteq\mrA$ are both $G$-invariant.
   Then if $\mrA$ is idempotent,
   also $\mrA\cdot_G\mrB$ is idempotent.
\end{corollary}

\begin{proof}
   We check that if $\mrA$ is idempotent so is $\mrA\cdot_G\mrB$

   \ceq{\hfill\big(\mrA\cdot_G\mrB\big)\ \cdot_G\ \big(\mrA\cdot_G\mrB\big)}
      {\subseteq}
      {\mrA\ \cdot_G\ \big(\mrA\cdot_G\mrB\big)}
      \hfill because $\mrA\cdot_G\mrB\subseteq\mrA$

   \ceq{}
      {\subseteq}
      {\big(\mrA\cdot_G\mrA\big)\cdot_G\mrB}
      \hfill by the lemma above

   \ceq{}
      {\subseteq}
      {\mrA\cdot_G\mrB}
\end{proof}

We show that, under the assumption of stationarity,
the operation $\cdot_G$ is associative.
The quotient map $\mrG\to\mrG/{\equiv_G}$ is almost a homomorphism.

\begin{proposition}\label{prop_orbits_main}
   Assume $\cnonfork_G$ is $1$-stationary, see 
   Definition~\ref{def_coheir_stationary}.
   Fix ${\mr a}\cnonfork_G{\mr b}$ arbitrarily.
   Then ${\mr a'}{\cdot}{\mr b'}\equiv_G{\mr a}{\cdot}{\mr b}$ for every 
   ${\mr a'}\equiv_G{\mr a}$ and  ${\mr b'}\equiv_G{\mr b}$ such that ${\mr a'}\cnonfork_G{\mr b'}$.
   Or, in other words,

   \ceq{\hfill({\mr a}{\cdot}{\mr b})_G}
      {=}
      {{\mr a}\cdot_G{\mr b}.}
\end{proposition}

\begin{proof}
   We prove two inclusions, only the second one requires stationarity.

   $\subseteq$ \ As ${\mr a}\cnonfork_G{\mr b}$ holds by hypothesis,
   ${\mr a}{\cdot}{\mr b}\in{\mr a}\cdot_G{\mr b}$.
   The inclusion follows by invariance.

   $\supseteq$ \ By invariance it suffices to show that the l.h.s.\@ contains
   ${\mr a}\cdot_G\{{\mr b}\}$.
   Let ${\mr a'}\in{\mr a}_G$ such that ${\mr a'}\cnonfork_G{\mr b}$.
   We claim that ${\mr a'}{\cdot}{\mr b}\in({\mr a}{\cdot}{\mr b})_G$.
   Both ${\mr a}$ and ${\mr a'}$ satisfy 
   ${\mr a}\equiv_G{\bl x}\cnonfork_G{\mr b}$.
   By $1$-stationarity, ${\mr a}\equiv_{G,\,{\mr b}}{\mr a'}$.
   Hence ${\mr a}{\cdot}{\mr b}\equiv_G{\mr a'}{\cdot}{\mr b}$.
\end{proof}

\begin{corollary}[(associativity)]\label{corol_orbits_associative}
   Assume $\cnonfork_G$ is $1$-stationary.
   Then for all $G$-invariant sets $\mrA$, $\mrB$ and  $\mrC$

   \ceq{\hfill\mrA\cdot_G\big(\mrB\cdot_G\mrC\big)}
      {=}
      {\big(\mrA\cdot_G\mrB\big)\cdot_G\mrC}.
\end{corollary}

\begin{proof}
   We can assume that $\mrA$, $\mrB$ and $\mrC$ are $G$-orbits.
   Say of ${\mr a}$, ${\mr b}$, and ${\mr c}$ respectively.
   We can assume that ${\mr a}\cnonfork_G{\mr b}{\cdot}{\mr c}$ and 
   ${\mr b}\cnonfork_G{\mr c}$.
   By Proposition~\ref{prop_orbits_main} the set on the l.h.s.\@ equals 
   $({\mr a}{\cdot}{\mr b}{\cdot}{\mr c})_G$.
   By a similar argument the set on the r.h.s.\@ equals 
   $({\mr a'}{\cdot}{\mr b'}{\cdot}{\mr c'})_G$ for some elements 
   ${\mr a'}$, ${\mr b'}$, and ${\mr c'}$.
   Proposition~\ref{prop_semi_associative} proves that inclusion 
   $\subseteq$ holds in general.
   But inclusion between orbits amounts to equality.
\end{proof}

The following lemma proves the existence of idempotent orbits.
The proof is self-contained, i.e.\@ it does not use Ellis's theorem on
the existence of idempotents in compact left topological semigroups
(however, the argument is very similar).
As a comparison, finding a proof in the setting of nonstandard analysis
is listed as an open problem in~\cite{Mauro2}.

\begin{lemma}\label{lem_Hindman}
   Assume $\cnonfork_G$ is $1$-stationary.
   If $\mrA$ is minimal among the idempotent sets that are type-definable
   over $G$, then $\mrA={\mr b}_G$ for some (any) ${\mr b}\in\mrA$.
   \end{lemma}
   \begin{proof}
   Fix arbitrarily some ${\mr b}\in\mrA$.
   By Corollary~\ref{corol_min_idempotent},
   the set $\mrA\cdot_G{\mr b}$ is contained in $\mrA$, idempotent and 
   type-definable over $G$ by Lemma~\ref{lem_typedef_Ab}.
   Therefore by minimality $\mrA\cdot_G{\mr b}=\mrA$.
   Let ${\mr\Aa'}\subseteq\mrA$ be the set of those ${\mr a}$ such that 
   ${\mr a}\cdot_G{\mr b}={\mr b}_G$.
   This set is non-empty because ${\mr b}\in\mrA\cdot_G{\mr b}$.
   It is easy to verify that ${\mr\Aa'}$ is type-definable over $G,{\mr b}$.
   As it is clearly invariant over $G$, it is type-definable over $G$.
   By associativity it is idempotent.
   Hence, by minimality, ${\mr\Aa'}=\mrA$.
   Then ${\mr b}\in{\mr\Aa'}$, which implies 
   ${\mr b}\cdot_G{\mr b}={\mr b}_G$.
   That is, ${\mr b}$ has idempotent orbit.
   Finally, by minimality, $\mrA={\mr b}_G$.
   \end{proof}

\begin{corollary}\label{corol_idempotent}
   Under the same assumptions of the lemma above, every 
   idempotent set that is  type-definable over $G$ contains 
   an element with an idempotent orbit.\QED
\end{corollary}
\section{Hindman's theorem}\label{Hindman}

In this section we merge the theory of idempotents presented in
Section~\ref{semigroups} with the proof of Ramsey's theorem to obtain 
Hindman's theorem.

Let ${\bar a}$ be a tuple of elements of $G$ of length $\le\omega$.
In Section~\ref{intro}\@ we defined $\fp\,{\bar a}$ and the notions 
of \cpawdot-closed and $\cpaw$-covered.
The relation $\cpaw$ is introduced mainly for future reference.
The classical Hindman's theorem is obtained with the positive integers
(as an additive semigroup) for $G$ and $<$ for $\cpaw$.

\theoremstyle{mio}
\newtheorem{Hindman}[thm]{Hindman Theorem}
\begin{Hindman}\label{thm_Hindman}
   Let $\cpaw$ be a relation on $G$ that makes it \cpawdot-closed 
   and $\cpaw$-covered.
   Then for every finite coloring of $G$ there is 
   a $\cpaw$-chain $\bar a$ such that $\fp\,\bar a$ is monochromatic.
   If there is no $g\in G$ such that $G\cpaw g$, we may further assume that the elements of the $\cpaw$-chain are all distinct.
\end{Hindman}

\begin{proof}
   We interpret $G$ as a structure in a language that extends the language of 
   semigroups with a symbol for $\cpaw$ and one for each subset of $G$.
   Let $\mrG$ be a saturated elementary superstucture of $G$.
   As observed in Remark~\ref{rk_coheir_stationary}, the language makes 
   $\cnonfork_G$ trivially $1$-stationary.

   We write ${\mr\G'}$ for the type-definable set $\{{\mr g} : G\cpaw{\mr g}\}$,
   which is non-empty because $G$ is $\cpaw$-covered.
   We claim that ${\mr\G'}$ is idempotent.
   In fact, if ${\mr a},{\mr b}\in{\mr\G'}$ then, as $G\cpaw{\mr a},{\mr b}$ 
   and ${\mr a}\cnonfork_G{\mr b}$, we must have that ${\mr a}\cpaw{\mr b}$.
   Therefore, from the \cpawdot-closedness of $G$ we infer 
   ${\mr a}{\cdot}{\mr b}\in{\mr\G'}$.

   Let ${\mr g_0}$ be an element of ${\mr\G'}$ with idempotent orbit as given by Corollary~\ref{corol_idempotent}.
   We can assume that ${\mr g_0}\notin G$ otherwise the sequence that is identically ${\mr g_0}$ trivially proves the theorem.
   If we want the elements of the chain $\bar{a}$ to be distinct it suffices require that ${\mr g_0}\notin G$. By definition of ${\mr g_0}$, this can be directly assumed when there is no $g\in G$ such that $G\cpaw g$.
   Let $p({\bl x})\in S(\mrG)$ be a global coheir of $\tp({\mr g_0}/G)$.
   Let ${\mr\bar g}$ be a coheir sequence of $p({\bl x})$, that is

   \ceq{\hfill {\mr g_i}}
   {\models}
   {p_{\restriction G,\,{\mr g_{\restriction i} }  }({\bl x}).}

   We write ${\mr \cev{g}_{\restriction i}}$ for the tuple 
   ${\mr g_{i-1}},\dots,{\mr g_{0}}$.
   By the idempotency of $({\mr g_0})_G$ and Proposition~\ref{prop_orbits_main},
   ${\mr h}\equiv_G{\mr g_0}$ for all 
   ${\mr h}\in \fp\,{\mr \cev{g}_{\restriction i}}$ and all $i$.
   It follows in particular that $\fp\,{\mr \cev{g}_{\restriction i}}$ is 
   monochromatic, say all its elements have color $1$.
   Now, we use the sequence ${\mr\bar g}$ to define $\bar a\in G^\omega$
   such that all elements of $\fp\,\bar a$ have color $1$.

   Assume as induction hypothesis that $\fp(a_{\restriction i},\,{\mr g_0})$
   is monochromatic of color $1$.
   Our goal is to find $a_i$ such that the same property holds for 
   $\fp(a_{\restriction i+1},\,{\mr g_0})$.

   First we claim that from the induction hypothesis it follows that, 
   for all $j$, all elements of 
   $\fp(a_{\restriction i},\,{\mr\cev{g}_{\restriction j}})$ have color $1$.
   In fact, the elements of 
   $\fp(a_{\restriction i},\,{\mr\cev{g}_{\restriction j}})$ have the form 
   $b\cdot{\mr h}$ for some $b\in\fp(a_{\restriction i})$ and 
   ${\mr h}\in\fp({\mr\cev{g}_{\restriction j}})$.
   As ${\mr h}\equiv_G{\mr g_0}$, we conclude that 
   $b\cdot{\mr h}\equiv_Gb\cdot {\mr g_0}$, which proves the claim.

   Let $\phi(a_{\restriction i},\,{\mr g_{i+1}},\,{\mr g_{\restriction i+1}})$ 
   say that all elements of 
   $\fp(a_{\restriction i},\,{\mr\cev{g}_{\restriction i+2}})$ have color $1$.
   As ${\mr\bar g}$ is a coheir sequence we can find  $a_i$ such that 
   $\phi(a_{\restriction i},\,a_i,\,{\mr g_{\restriction i+1}})$.
   Hence all elements of 
   $\fp(a_{\restriction i+1},\,{\mr\cev{g}_{\restriction i+1}})$ 
   have color $1$.
   Therefore $a_i$ is as required.
\end{proof}

Hindman's theorem generalizes to a proposition that subsumes Ramsey's theorem.
It is usually referred to as the Milliken–Taylor 
theorem~\cite{Milliken} and~\cite{Taylor}.
By the following observation, we may use virtually the same proof.

\begin{proposition}\label{prop_Hindman2}
   Assume $\cnonfork_G$ is $1$-stationary.
   Let ${\mr\bar g}\in\mrG^{\omega}$ be a coheir sequence 
   of some global coheir of $\tp({\mr g}/G)$ where ${\mr g}$ has idempotent orbit.
   Let ${\mr\bar h}\in\mrG^{\omega}$ be such that 
   ${\mr h_i}\in\fp({\mr\cev{g}_{\restriction I_i}})$ for some finite non-empty
   $I_i\subseteq\omega$ such that  $I_i<I_{i+1}$.
   Then ${\mr\bar h}\equiv_G{\mr\bar g}$.
\end{proposition}

\begin{proof}
   Write $n_i$ for the minimum of $I_i$.
   It suffices to prove that 
   ${\mr h_i}\equiv_{G,{\mr g_{\restriction n_i}}}{\mr g_{n_i}}$.
   Note that the type 
   ${\mr g}\equiv_G{\bl x}\cnonfork_G{\mr g_{\restriction n_i}}$
   is satisfied both by ${\mr h_i}$ and ${\mr g_{n_i}}$,
   hence the claim follows by stationarity.
\end{proof}

Write \emph{$\fp(\bar a)_n$\/} for the $n$-uniform 
hypergraph with vertex set $\fp(\bar a)$ and as edges those sets 
$\{h_1,\dots,h_n\}$ such that $h_i\in\fp(a_{\restriction I_i})$ for some
finite sets $I_1<\dots<I_n$.
\theoremstyle{mio}
\newtheorem{MillikenTaylor}[thm]{Milliken-Taylor Theorem}
\begin{MillikenTaylor}\label{thm_MillikenTaylor}
   Let $\cpaw$ be a relation on $G$ that makes it \cpawdot-closed 
   and $\cpaw$-covered.
   Then for every positive integer $n$ and every finite coloring of the
   complete $n$-uniform hypergraph with vertex set $G$ there is a 
   $\cpaw$-chain $\bar a$ such that $\fp(\bar a)_n$ is monochromatic.\QED
\end{MillikenTaylor}

\begin{proof}
   Given a coheir sequence ${\mr\bar g}$ as in the proof of 
   Theorem~\ref{thm_Hindman} we want to define 
   $\bar a\in G^\omega$ such that $\fp(\bar a)_n$ is monochromatic.
   By the proposition above, $\fp({\mr\cev{g}_{\restriction i}})_n$
   is monochromatic for every $i\ge n$.
   As in the proof of Theorem~\ref{thm_Hindman}, we define by induction $\bar a\in G^\omega$ in such a way that
    $\fp(a_{\restriction i},\,{\mr\cev{g}_{\restriction n}})_n$
   is a finite monochromatic subgraph of $G$.
\end{proof}
\section{The Hales-Jewett theorem}\label{HJ}

\def\medrel#1{\parbox[t]{6ex}{$\displaystyle\hfil #1$}}
\def\ceq#1#2#3{\parbox[t]{39ex}{$\displaystyle #1$}
               \medrel{#2}
               {$\displaystyle #3$}}

The Hales-Jewett theorem is a purely combinatorial statement 
that implies the van der Waerden theorem.
The original proof by Alfred Hales and Robert Jewett is 
combinatorial~\cite{HJ}.
An alternative proof, also combinatorial, is due by 
Saharon Shelah~\cite{Shelah}.
Our proof is similar to the proof by Andreas Blass in~\cite{Blass} 
(based on ideas from~\cite{BBH}), but we use saturated models where 
he uses Stone-\v{C}ech compactification.
We present three versions of the main theorem.

First we prove an abstract algebraic version due to 
Sabine Koppelberg~\cite{Koppelberg} which is easier to state and to prove 
(this version comes in two variants).
The classical version follows easily from the algebraic one.

We work with the same notation as in Section~\ref{semigroups}.
We say that an element ${\mr c}$ is \emph{left-minimal\/} (w.r.t.\@ $\mrA$)
if ${\mr c}\in\mrA\cdot_G{\mr g}$ for every ${\mr g}\in\mrA\cdot_G{\mr c}$.

\begin{proposition}\label{prop_minimal_existence1}
   Assume $\cnonfork_G$ is $1$-stationary.
   Let $\mrA$ be idempotent and type-definable over $G$.
   Then $\mrA$ contains a left-minimal element 
   ${\mr c}$ with idempotent orbit.
\end{proposition}

\begin{proof} Construct by induction a chain of type-definable idempotent sets $\mrB_\alpha\subseteq \mrA$ and elements ${\mr b}_\alpha\in \mrB_\alpha $ such that $\mrB_{0}=\mrA$ and $\mrB_{\alpha+1}=\mrA\cdot_G {\mr b}_\alpha$. For $\alpha$ limit take the intersection.  By idempotency of $\mrA$, it is straightforward to check that  $\mrB_{\alpha+1}\subseteq\mrB_\alpha$. The sets  $\mrB_\alpha$ are type-definable and idempotent by \ref{lem_typedef_Ab} and \ref{corol_min_idempotent}. For $\alpha$ limit $\mrB_\alpha$ is non-empty by compactness, as it is intersection of a chain of closed sets. 
	
	For some $\alpha$ we cannot properly extend this construction. For this $\alpha$, for every ${\mr c},{\mr g} \in \mrB_\alpha$ we have
	$\mrA\cdot_G {\mr c} =\mrB_\alpha=\mrA\cdot_G {\mr g} $ .
	Hence every $ {\mr c} \in\mrB_\alpha $ is left-minimal.
	As $\mrB_\alpha$ is idempotent, by Corollary~\ref{corol_idempotent}
	there is some $ {\mr c} \in\mrB_\alpha$ with idempotent orbit.
\end{proof}

\begin{proposition}\label{prop_minimal_existence2}
   Assume $\cnonfork_G$ is $1$-stationary.
   Let $\mrA$ be idempotent and type-definable over $G$.
   Let ${\mr c}_G$ be idempotent and such that 
   ${\mr c}\cdot_G\mrA,\ \mrA\cdot_G{\mr c}\subseteq\mrA$.
   Then
   \begin{itemize}
   \item[1.] ${\mr c}\cdot_G\mrA\cdot_G{\mr c}$ contains some ${\mr g}$
             with idempotent orbit; 
   \item[2.] if moreover ${\mr c}$ is left-minimal, then 
             ${\mr c}\equiv_G{\mr g}$ for every ${\mr g}$ as in \ssf{1}.
   \end{itemize} 
\end{proposition}

Note, parenthetically, that the set in \ssf{1} may not be type-definable, 
therefore Corollary~\ref{corol_idempotent} does not apply 
directly and we need an indirect argument.

\begin{proof}
   \ssf{1.} \ From ${\mr c}\cdot_G\mrA\subseteq\mrA$ we obtain that 
   $\mrA\cdot_G{\mr c}$ is idempotent.
   As it is also type-definable, $\mrA\cdot_G{\mr c}$ contains 
   a ${\mr b}$ with idempotent orbit by Corollary~\ref{corol_idempotent}.
   There is an ${\mr a}\in\mrA$ such that ${\mr b}_G={\mr a}\cdot_G{\mr c}$, then ${\mr b}\cdot_G{\mr c}={\mr b}_G$.
   From this we obtain that 
   ${\mr c}\cdot_G{\mr b}$ is idempotent and contained in 
   ${\mr c}\cdot_G\mrA\cdot_G{\mr c}$.

   \ssf{2.} \ From ${\mr g}\in{\mr c}\cdot_G\mrA\cdot_G{\mr c}$ and the 
   idempotency of ${\mr c}_G$ we obtain ${\mr g}_G={\mr c}\cdot_G{\mr g}$.
   As ${\mr g}\in\mrA\cdot_G{\mr c}$, from the left-minimality of ${\mr c}_G$ 
   we obtain ${\mr c}\in\mrA\cdot_G{\mr g}$.
   Hence ${\mr c}_G={\mr c}\cdot_G{\mr g}$, by the idempotency of ${\mr g}_G$.
   Therefore ${\mr c}_G={\mr g}_G$, which proves \ssf{2}.
\end{proof}

The following is a technical lemma that is required in many proofs below.

\begin{proposition}\label{prop_HJ_tecnical}
   Assume $\cnonfork_G$ is $1$-stationary.
   Let $\sigma:\mrG\to\mrG$ be a semigroup homomorphism definable over $G$.
   Then for every ${\mr a}, {\mr b}\in\mrG$
   \begin{itemize}
   \item[1.] $\sigma\big[{\mr a}_G\big]
             \ =\ 
             (\sigma\,{\mr a})_G$

   \item[2.] $\sigma\big[{\mr a}\cdot_G{\mr b}\big]
             \ =\ 
             \sigma\,{\mr a}\cdot_G\sigma\,{\mr b}$.
   \end{itemize}
\end{proposition}

\begin{proof}\ \ssf{1.}  
   As ${\mr a}\equiv_G{\mr a'}$ implies 
   $\sigma\,{\mr a}\equiv_G\sigma\,{\mr a'}$, inclusion $\subseteq$ is clear.
   For the converse, note that the type 
   $\E{\bl y}\,\big[\sigma\,{\bl y}={\bl x}
      \,\wedge\,
      {\bl y}\equiv_G{\mr a}\big]$ 
   is trivially realized by $\sigma\,{\mr a}$.
   Therefore it is realized by all elements of $(\sigma\,{\mr a})_G$.
   Hence all elements of $(\sigma\,{\mr a})_G$ are the image of 
   some element in ${\mr a}_G$.

   \ssf{2.} \ 
   Let ${\mr a}\equiv_G{\mr a'}\cnonfork_G{\mr b'}\equiv_G{\mr b}$.
   By Proposition~\ref{prop_orbits_main} we have 
   $\sigma\big[{\mr a}\cdot_G{\mr b}\big]
   =\sigma\big[({\mr a'}\cdot{\mr b'})_G\big]$.
   Then it suffices to prove that 
   $\sigma\big[({\mr a'}\cdot{\mr b'})_G\big]
   \subseteq\sigma\,{\mr a}\cdot_G\sigma\,{\mr b}$, 
   because by \ssf{1} and Proposition~\ref{prop_orbits_main} 
   both sides of the equality are orbits.
   As $\sigma$ preserves $\cnonfork_G$ 
   and orbits, we obtain that $\sigma({\mr a'}\cdot{\mr b'})$ 
   is in $\sigma\,{\mr a}\cdot_G\sigma\,{\mr b}$, as well as all 
   other elements of $\sigma\big[({\mr a'}\cdot{\mr b'})_G\big]$.
\end{proof}

\theoremstyle{mio}
\newtheorem{HalesJewett}[thm]{Hales-Jewett Theorem}
\begin{HalesJewett}[(Koppelberg's version)]\label{thm_abstract_HJ}
   Let $G$ be an infinite semigroup and let $C\subset G$
   be a nice subsemigroup.
   Let $\Sigma$ be a finite set of retractions of $G$ onto $C$.
   Then, for every finite coloring of $C$, there is an $a\in G\sm C$ 
   such that $\{\sigma\,a:\sigma\in\Sigma\}$ is monochromatic.
\end{HalesJewett}

\begin{proof}
   Let $G\preceq\mrG$.
   Here $\mrG$ is a monster model in a language that expands the natural 
   one with a symbol for all subsets of $G$ and for every retraction in $\Sigma$.
   As observed in Remark~\ref{rk_coheir_stationary}, 
   this makes $\cnonfork_G$ trivially $1$-stationary.
   Let $\mrC$ be the definable set such that $C=G\cap\mrC$.
   By elementarity, $\mrC$ is a nice subsemigroup of $\mrG$.
   The language contains also symbols for 
   the retractions $\sigma:\mrG\to\mrC$.
   
   By Proposition~\ref{prop_minimal_existence1}, there is a left-minimal
   ${\mr c}\in\mrC$ with idempotent orbit.
   
   By niceness, $\mrG\sm\mrC$ and ${\mr c}$ satisfy the assumptions of 
   Proposition~\ref{prop_minimal_existence2}.
   Hence, by the first claim of that proposition, there is an idempotent
   ${\mr g}\in{\mr c}\cdot_G(\mrG\sm\mrC)\cdot_G{\mr c}$.
   In particular, ${\mr g}\in\mrG\sm\mrC$.
   Now apply the second claim of Proposition~\ref{prop_HJ_tecnical}, 
   with $\mrC$ for $\mrA$ to obtain 
   $\sigma\,{\mr g}\in{\mr c}\cdot_G\mrC\cdot_G{\mr c}$ 
   for all $\sigma\in\Sigma$.
   As $\sigma\,{\mr g}$ is also idempotent, we apply 
   Proposition~\ref{prop_minimal_existence2} to conclude that 
   $\sigma\,{\mr g}\equiv_G{\mr c}$.
   In particular the set $\{\sigma\,{\mr g}:\sigma\in\Sigma\}$ is
   monochromatic.
   
   Though the element ${\mr g}$ above need not belong to $G\sm C$,
   by elementarity $G\sm C$ contains some $a$ with the same property
   and this proves the theorem.
\end{proof}

Finally we show how the classical Hales-Jewett theorem follows 
from its abstract version.

If $C$ and $X$ are two semigroups we denote by $C*X$ their free product.
That is, $C*X$ contains finite sequences of elements of $C\cup X$, 
below called \emph{words,} that alternate elements in $C$ with
elements in $X$.
The product of two words is obtained concatenating them and, when it 
applies, replacing two contiguous elements of the same semigroup by their
product.
Note that $C$ and $X$ are nice subsemigroups of $C*X$.
When $X$ is the free semigroup generated by a variable $x$, we denote $C*X$ by \emph{$C[x]$.} If $w(x)$ is an element of $C[x]$ and $a\in C$ we denote by $w(a)$ the result of replacing $x$ by $a$ in $w(x)$.


\begin{HalesJewett}[(classical version)]\label{thm_HalesJewett}
   Let $C$ be a semigroup generated by some finite set $A$.
   Let $x$ be a variable.
   Then for every finite coloring of $C[x]$ there is a $w(x)\in C[x]\sm C$
   such that $\{ w(a): a\in A\}$ is monochromatic.
\end{HalesJewett}

\begin{proof}
   Let $G=C[x]$.
   For every $a\in A$ the homomorphism $\sigma_a:w(x)\mapsto w(a)$ 
   is a retraction of $G$ onto $C$.
   Hence we can apply the theorem above.
\end{proof}

We conclude with a variant of Theorem~\ref{thm_abstract_HJ} 
that applies to a broader class of semigroup homomorphisms.
This result is not required for the following.

For $\Sigma$ a set of maps $\sigma:G\to C$ and $c\in C$ we define 

\ceq{\hfill\emph{$\Sigma^{-1}[c]$}}
   {=}
   {\bigcap_{\sigma\in\Sigma}\sigma^{-1}[c]}\smallskip

Clearly, when the maps in $\Sigma$ are retractions, $\Sigma^{-1}[c]$ is 
non-empty for all $c\in C$ because it contains at least $c$.

\begin{HalesJewett}[(yet another variant)]\label{thm_hom_HJ}
   Let $C$ be a semigroup and let $\Sigma$ be a finite set of homomorphisms 
   $\sigma:G\to C$ such that $\Sigma^{-1}[c]$ is non-empty for all $c\in C$.
   Then, for every finite coloring of $C$, there is a $g\in G$ such that
   the set $\{\sigma\,g\ :\ \sigma\in\Sigma\}$ is monochromatic.
\end{HalesJewett}

\begin{proof}
   Let $G*C$ be the free product of the two semigroups.
   Any homomorphism $\sigma:G\to C$ extends canonically to a retraction
   of $G*C$ onto $C$.
   The elements of $G$ that occur in a word are replaced by their image under
   $\sigma$, finally the elements in the resulting sequence are multiplied.
   This extension is denoted by the same symbol $\sigma$.

   Apply Theorem~\ref{thm_abstract_HJ} to obtain some $w\in G*C$ such that 
   $\{\sigma \,w:\sigma\in\Sigma\}$ is monochromatic.
   Suppose $w=c_0\cdot g_0\cdots\cdots c_n\cdot g_n$ for some $g_i\in G$ and 
   $c_i\in C$, where one or both of $c_0$ or $g_n$ could be absent.
   Pick some $h_i\in\Sigma^{-1}[c_i]$ and let 
   $g=h_0\cdot g_0\cdots\cdots h_n\cdot g_n$.
   Then $\{\sigma\,g:\sigma\in\Sigma\}$ is monochromatic as required 
   to complete the proof.
\end{proof}
\section{Carlson's theorem}\label{carlson}

\def\medrel#1{\parbox[t]{6ex}{$\displaystyle\hfil #1$}}
\def\lbox#1{\parbox[t]{7ex}{$\displaystyle #1$}}
\def\rbox#1{{$\displaystyle #1$}}
\def\ceq#1#2#3{\lbox{#1}\medrel{#2}\rbox{#3}}

This section is devoted to the following lemma and some of its consequences.

\begin{lemma}\label{lem_Carlson}
   Let $\Sigma$ be a finite set of retractions of $G$ onto 
   a nice subsemigroup $C$.
   Let $\cpaw$ be a relation on $G$ that makes it \cpawdot-closed 
   and $\cpaw$-covered by $G\sm C$.
   Then, for every finite coloring of $G$, there is a $\cpaw$-chain 
   $\bar a\in  (G\sm C)^\omega$ such that $\fp^{\Sigma}\bar a\sm C$
   is monochromatic.
\end{lemma}

\begin{proof}
   The models ${\mr\G}$ and ${\mr\C}$ are as in the proof of
   Theorem~\ref{thm_abstract_HJ}.
   The language is the same with $\cpaw$ included.
   Let ${\mr\B}=\{{\mr g}\in\mrG\sm\mrC: G\cpaw{\mr g}\}$.
   By Proposition~\ref{prop_minimal_existence1} there is some 
   left-minimal ${\mr c}\in\mrC$ with idempotent orbit.
   As $G$ is $\cpaw$-covered by $G\sm C$, the set $\mrB$ is non-empty.
   As $G$ is \cpawdot-closed and $C$ is nice, $\mrB$ and ${\mr c}$ satisfy
   the assumptions of Proposition~\ref{prop_minimal_existence2}.
   Then, ${\mr c}\cdot_G\mrB\cdot_G{\mr c}$ contains some ${\mr g_0}$ with idempotent orbit.
   By Proposition~\ref{prop_HJ_tecnical}, we obtain that
   $\sigma\,{\mr g_0}\in{\mr c}\cdot_G\mrC\cdot_G{\mr c}$ 
   for all $\sigma\in\Sigma$.
   As $(\sigma\,{\mr g_0})_G$ is also idempotent, we apply the second claim of
   Proposition~\ref{prop_HJ_tecnical}, with $\mrC$ for $\mrA$ to conclude 
   that    $\sigma\,{\mr g_0}\equiv_G{\mr c}$ for all $\sigma\in\Sigma$.
   Now, let ${\mr\bar g}$ be a coheir sequence as in 
   Theorem~\ref{thm_Hindman}, and assume the notation thereof.
   As ${\mr g_0}\in{\mr c}\cdot_G\mrB\cdot_G{\mr c}$ then 
   ${\mr c}\cdot_G {\mr g_0} = {\mr g_0}\cdot_G{\mr c} = ({\mr g_0})_G$.
   Hence ${\mr h}\equiv_G{\mr g_0}$ for all $i$ and all 
   ${\mr h}\in \fp\,{\mr \cev{g}_{\restriction i}}\sm \mrC$.
   In particular all these ${\mr h}$ have the same color, say color $1$.
   Now, we can use the sequence ${\mr\bar g}$ to define 
   $\bar a\in (G\sm C)^\omega$ such that all elements of 
   $\fp^\Sigma\,\bar a\sm C$ have color $1$ by reasoning as in the proof 
   of Theorem~\ref{thm_Hindman}.
\end{proof}

Carlson's theorem is a result that combines the theorems of Hindman and 
Hales-Jewett and has a number of important consequences. 
We refer the reader to \cite{DK} for a discussion of some of 
these consequences.
The definitions in Example~\ref{ex_words} will help matching the notation.

We first present a Koppelberg-style version of the theorem.
It is obtained from the lemma above applying a suitable coding.

\theoremstyle{mio}
\newtheorem{Carlson}[thm]{Carlson Theorem}
\begin{Carlson}[(à la Koppelberg)]\label{thm_Carlson2}
   Let $\Sigma$ be a finite set of retractions of $G$
   onto a nice subsemigroup $C$.
   Let $\bar s\in(G\sm C)^\omega$.
   Then for every finite coloring of $G$, there is an increasing sequence of positive integers $\<n_i:i<\omega\>$ and some $a_i\in\fp^{\Sigma}s_{\restriction\,[n_i,n_{i+1})}\sm C$ such that $\fp^{\Sigma}\bar a\sm C$ is monochromatic.
\end{Carlson}

\begin{proof}
   Let $G_*$ be the free semigroup generated by the alphabet

   \hfil$\{\<\sigma,g\>\ :\ \sigma\in\Sigma\cup\{\id_G\},\ g\in G\sm C\}$. 
   
   The semigroup $C_*$ is defined as $G_*$, only $\sigma$ is restricted 
   to range over $\Sigma$.
   Clearly $C_*$ is a nice subsemigroup of $G_*$.
   We associate to each $\sigma\in\Sigma$ the endomorphism of $G_*$ that
   substitutes $\sigma$ for every occurrence of $\id_G$ in a word.
   These maps, which we denote by $\sigma_*$, are retractions 
   of $G_*$ onto $C_*$.
   
   If $g_*\in G_*$ has the form $\<\sigma_1,g_1\>\cdots\<\sigma_n,g_n\>$ 
   we call $\sigma_1\,g_1\cdots\sigma_n\,g_n\in G$ the 
   \textit{evaluation\/} of $g_*$.
   We denote the evaluation by ${\rm eval}(g_*)$.
   As $\tau\,\sigma=\sigma$ for every $\tau,\sigma\in\Sigma$, we have that
   ${\rm eval}(\sigma_*\,g_*)=\sigma\,{\rm eval}(g_*)$.
   The evaluation of $g_*\in C_*$ belongs to $C$ and,
   as $C$ is nice, the evaluation of $g_*\in G_*\sm C_*$ belongs to $G\sm C$.

   We color each element of $G_*$ with the color of its evaluation.

   We define the relation $\cpaw$ on $G_*$.
   First, we need to define the \textit{well-formed\/} elements of $G_*$.
   These are elements of the form 
   $\<\sigma_1,s_{i_1}\>\cdots\<\sigma_n,s_{i_n}\>$ for some $i_1<\dots<i_n$.
   Now, for $h_*,g_*\in G_*$ we define $h_*\cpaw g_*$ if one 
   of the following holds
   \\
   \ssf{1.}\quad $h_*$ is not well-formed while $g_*$ is;
   \\
   \ssf{2.}\quad the product (i.e., concatenation) $h_*g_*$ is well-formed.
   \\
   It is immediate to verify that $\cpaw$ is $G_*$ is \cpawdot-closed 
   and $\cpaw$-covered by $G_*\sm C_*$.
   Therefore by Lemma~\ref{lem_Carlson} there is a $\cpaw$-chain 
   $\bar a_*\in (G_*\sm C_*)^\omega$ such that $\fp^{\Sigma}\bar a_*\sm C_*$ 
   is monochromatic.
   We can assume that all elements of $\bar a_*$ are well-formed (only 
   the first element might be ill-formed, but we can drop it).
   Then the sequence $ \<{\rm eval}( a_{i*}) :i\in \omega\>$ is as required by the lemma.
\end{proof}

From the algebraic version of Carlson's theorem we obtain the 
classical one in the same way as for the Hales-Jewett theorem 
(Theorem~\ref{thm_HalesJewett}), which we refer to for the notation.

\begin{corollary}[(Carlson's theorem, classical version)]\label{corol_Carlson}
   Let $C$ be a semigroup generated by some finite set $A$.
   Let $x$ be a variable.
   Let $\bar s\in\big(C[x]\sm C\big)^\omega$.
   Let $\Sigma$ contain, for every $a\in A$, the function $w(x)\mapsto w(a)$.
   Then, for every finite coloring of $C[x]$, there is an increasing sequence of positive integers $\<n_i:i<\omega\>$ and some $a_i\in\fp^{\Sigma}s_{\restriction\,[n_i,n_{i+1})}\sm C$ such that $\fp^{\Sigma}\bar a\sm C$ is monochromatic (with the terminology of Example~\ref{ex_words}, $\bar a$ is an extracted variable sequence of $\bar s$).\QED
\end{corollary}
\section{Gowers's partition theorem}\label{gowers}

\def\medrel#1{\parbox[t]{6ex}{$\displaystyle\hfil #1$}}
\def\lbox#1{\parbox[t]{20ex}{$\displaystyle #1$}}
\def\rbox#1{{$\displaystyle #1$}}
\def\ceq#1#2#3{\lbox{#1}\medrel{#2}\rbox{#3}}

The following is similar to Lemma~\ref{lem_Carlson} but here $\Sigma$ contains compositions of homomorphisms.

\begin{lemma}\label{lem_Gowers}
   For $0<i<n$, let $G_i$ be a nice subsemigroup of $G_{i+1}$ and let $\sigma_i:G_{i+1}\to G_i$ be homomorphisms.
   Let $\cpaw$ be a relation on $G_n$ that makes it \cpawdot-closed and  $\cpaw$-covered by $G_n\sm G_{n-1}$.
   Finally, let $\Sigma=\big\{\sigma_i\circ\dots\circ\sigma_{n-1}: 0< i< n\big\}$.
   Then, for every finite coloring of $G_n$, there is a $\cpaw$-chain 
   $\bar a\in \big(G_n\sm G_{n-1}\big)^\omega$ such that $\fp^{\Sigma}\bar a\sm G_{n-1}$ is monochromatic.
\end{lemma}

\begin{proof}
   For convenience, we let $i$ run from $0$, hence we agree that $\sigma_0:G_1\to G_0=G_1$ is the identity.
   Let ${\mr\B_n}=\{{\mr b}\in{\mr\G_n}\sm{\mr\G_{n-1}}\,:\,G_n\cpaw{\mr b}\}$ and ${\mr\B_i}=\sigma_i[{\mr\B_{i+1}}]$.
   Note that the ${\mr\B_i}$ are non-empty because $G_n$ is $\cpaw$-covered by $G_n\sm G_{n-1}$.
   Also, as ${\mr\G_i}$ is a nice subsemigroup of ${\mr\G_{i+1}}$, we have that ${\mr\B_i}\cdot_G{\mr\B_{i+1}},\ {\mr\B_{i+1}}\cdot_G{\mr\B_i}\subseteq{\mr\B_{i+1}}$.
   
   We claim there is some ${\mr b_n}\in{\mr\B_n}$ with idempotent orbit such that, if we define ${\mr b_i}=\sigma_i\,{\mr b_{i+1}}$ for $0\le i< n$, the following holds
   
   \ceq{\hfill{\mr b_n}\cdot_G{\mr b_i}}
   {=}
   {{\mr b_i}\cdot_G{\mr b_n}}
   \medrel{=}
   \rbox{({\mr b_n})_G.}
   
   Note that these equalities may be replaced by
   
   \ceq{\sharp_i\hfill{\mr b_i}\cdot_G{\mr b_{i+1}}}
   {=}
   {{\mr b_{i+1}}\cdot_G{\mr b_i}}
   \medrel{=}
   \rbox{({\mr b_{i+1}})_G.}

   Let ${\mr b_0}={\mr b_1}$ be any element of ${\mr\B_0}$ with idempotent orbit.
   We assume as induction hypothesis that we have ${\mr b_i}\in{\mr\B_i}$ for $i\le k$, with idempotent orbits, such that ${\mr b_i}=\sigma_i\,{\mr b_{i+1}}$ and $\sharp_i$ hold for all $i<k$.
   We show how to find ${\mr b_{k+1}}$.
   
   We prove that ${\mr b_k}$ and the set ${\mr\B_{k+1}}\cap\sigma_k^{-1}[{\mr b_k}]$, which below we denote by $\mrA$ for short, satisfy the assumptions of Proposition~\ref{prop_minimal_existence2}.
   The proof of the idempotency of $\mrA$ is left to the reader.
   We prove that ${\mr b_k}\cdot_G\mrA\,\subseteq\, \mrA$, the proof of $\mrA\cdot_G{\mr b_k}\,\subseteq\, \mrA$ is similar.
   As ${\mr b_k}\cdot_G{\mr\B_{k+1}}\,\subseteq\,{\mr\B_{k+1}}$ by nicety, it suffices to prove that ${\mr b_k}\cdot_G\sigma_k^{-1}[{\mr b_k}]$ is contained in $\sigma_k^{-1}[{\mr b_k}]$.
   This latter inclusion holds because, by the induction hypothesis,
   
   \ceq{\hfill\sigma_k\Big[{\mr b_k}\cdot_G\sigma_k^{-1}[{\mr b_k}]\Big]}
   {=}
   {\sigma_k[{\mr b_k}]\cdot_G{\mr b_k}}
   \medrel{=}
   \rbox{{\mr b_{k-1}}\cdot_G{\mr b_k}}
   \medrel{=}
   \rbox{({\mr b_k})_G.}

   Now we apply Proposition~\ref{prop_minimal_existence2} to find an idempotent ${\mr b_{k+1}}\in {\mr b_k}\cdot_G\mrA\cdot_G{\mr b_k}$.
   Therefore $\sharp_k$ is satisfied.
   Moreover $\sigma_k\,{\mr b_{k+1}}\in ({\mr b_k})_G$ by Proposition~\ref{prop_HJ_tecnical}, hence we can assume ${\mr b_k}=\sigma_k\,{\mr b_{k+1}}$ as claimed above.

   Finally, as in the proof of Theorem~\ref{thm_Hindman}, the required chain $\bar a$ is obtained from a coheir sequence of a global coheir of $\tp({\mr b_n}/G)$. 
\end{proof}

\begin{remark}\label{rem_Gowers}
  The lemma above continues to hold, with essentially the same proof, if for $\Sigma$ we take a set of the form 
  
\ceq{\hfill\Sigma}{=}{\bigcup^{n-1}_{i=1}\ \Sigma_i\circ\dots\circ\Sigma_{n-1}}

   where

   \ceq{\hfill\Sigma_i\circ\dots\circ\Sigma_{n-1}}{=}{\Big\{\sigma_i\circ\dots\circ\sigma_{n-1}\ :\  \sigma_i\in\Sigma_i,\dots,\sigma_{n-1}\in\Sigma_{n-1}\Big\}}

   and where $\Sigma_i$ are some finite sets of homomorphisms $G_{i+1}\to G_i$ such that for every $g\in G_i$ the set $\Sigma_i^{-1}[g]$ is non-empty.\QED
\end{remark}

Let $G_i$ be the set of functions $a:\omega\to\{0,\dots,i\}$ with finite \emph{support} that is, the set \emph{$\supp(a)$}$=\{x\in\omega\,:\,a\,x\neq0\}$ is finite.
We introduce a semigroup operation on $G_i$ by defining 
$(a{\cdot}b)\,x=\max\{ax,bx\}$.
This makes $G_i$ a nice subsemigroup of $G_{i+1}$.

\begin{corollary}[(Gowers Partition Theorem)]\label{corol_Gowers}
   With $G_i$ as above, let $\sigma_i:G_{i+1}\to G_i$ be homomorphisms and let $\Sigma$ be as in Lemma~\ref{lem_Gowers}.
   Then for every finite coloring of $G_n$ there is an 
   $\bar a\in\big(G_n\sm G_{n-1}\big)^\omega$ such that $\fp^{\Sigma}\bar a\sm G_{n-1}$ is monochromatic and $\supp(a_i)<\supp(a_{i+1})$.
\end{corollary}

The homomorphisms $\sigma_i$ usually considered in the literature are so-called \textit{tetris\/} operations i.e.\@ $(\sigma_i\,a)x=\max\{a\,x-1,\ 0\}$, or generalizations thereof.
However the theorem is more general.
%

\begin{proof}
   Let $\cpaw$ be the relation $\supp(a)<\supp(b)$ and apply Theorem~\ref{lem_Gowers}.
\end{proof}


\end{document}